\documentclass[12pt]{article}
\usepackage{amsthm,amsmath,amssymb}
\usepackage{graphicx}
\usepackage[shortlabels]{enumitem}
\usepackage{multirow}
\usepackage{hyperref}
\usepackage{url}
\usepackage{longtable}
\usepackage{tikz}

\theoremstyle{plain}
\newtheorem{theorem}{Theorem}[section]

\newtheorem{proposition}[theorem]{Proposition}
\newtheorem{lemma}[theorem]{Lemma}

\theoremstyle{definition}

\theoremstyle{remark}

\newtheorem{example}{Example}

\makeatletter
\newcommand{\subjclass}[2][1991]{%
  \let\@oldtitle\@title%
  \gdef\@title{\@oldtitle\footnotetext{#1 \emph{Mathematics subject classification.} #2}}%
}
\newcommand{\keywords}[1]{%
  \let\@@oldtitle\@title%
  \gdef\@title{\@@oldtitle\footnotetext{\emph{Key words and phrases.} #1.}}%
}
\makeatother

\title{Another construction of edge-regular graphs with regular cliques}

\author{Gary R. W. Greaves\thanks{Supported by  the Singapore
Ministry of Education Academic Research Fund (Tier 1);  grant number:  RG127/16.}\\
  School of Physical and Mathematical Sciences, \\
  Nanyang Technological University, \\
   21 Nanyang Link, Singapore 637371\\
  {\tt grwgrvs@gmail.com}
\and
  Jack H. Koolen\thanks{Partially supported by the National Natural Science Foundation of China
(No. 11471009 and No.11671376).}\\
  Wen-Tsun Wu Key Laboratory of CAS,\\
  School of Mathematical Sciences, \\
  University of Science and Technology of China, \\
  Hefei, Anhui, 230026, P.R. China\\
  {\tt koolen@ustc.edu.cn}
}



\begin{document}
	
\maketitle

\begin{abstract}
	We exhibit a new construction of edge-regular graphs with regular cliques that are not strongly regular.
	The infinite family of graphs resulting from this construction includes an edge-regular graph with parameters $(24,8,2)$.
	We also show that edge-regular graphs with $1$-regular cliques that are not strongly regular must have at least $24$ vertices.
\end{abstract}

\section{Introduction and definitions}

We begin with various definitions of regularity.
A $v$-vertex graph $\Gamma$ is called \textbf{$k$-regular} if there exists a $k$ such that each vertex of $\Gamma$ has degree $k$.
A $v$-vertex, $k$-regular non-empty graph is called \textbf{edge-regular} with parameters $(v,k,\lambda)$ if 
every pair of adjacent vertices $x \sim y$ have $\lambda$ common neighbours.
A $(v,k,\lambda)$-edge-regular graph is called \textbf{strongly regular} with parameters $(v,k,\lambda,\mu)$ if 
every pair of distinct nonadjacent vertices $x \not \sim y$ have $\mu$ common neighbours.
A clique $\mathcal C$ is called \textbf{regular} (or $e$-regular) if every vertex not in $ \mathcal C$ is adjacent to a constant number $e>0$ of vertices in $\mathcal C$.

Recently, the authors~\cite{Greaves:2017qy} provided an infinite family of non-strongly-regular, edge-regular graphs having regular cliques, thus answering a question of Neumaier~\cite[Page 248]{Neu:regCliques}.
The smallest graph in this family has parameters $(28,9,2)$.
In this note we offer a new construction that gives rise to a non-strongly-regular, edge-regular graph $\mathcal G$ with parameters $(24,8,2)$ having a $1$-regular clique.
In fact, $\mathcal G$ is isomorphic to one of the four examples found by Goryainov and Shalaginov~\cite{Goryainov14:Deza}.
Recently, however, Evans et al.~\cite{EGP} discovered an edge-regular, but not strongly regular graph on $16$ vertices that has $2$-regular cliques with order $4$, and proved that, up to isomorphism, this is the unique edge-regular, but not strongly regular graph on at most $16$ vertices having a regular clique.

A graph $\Gamma$ of diameter $d$ is called \textbf{$a$-antipodal} if the relation of being at distance $d$ or distance $0$ is an equivalence relation on the vertices of $\Gamma$ with equivalence classes of size $a$.
A set of cliques of a graph $\Gamma$ that partition the vertex set of $\Gamma$ is called a \textbf{spread} in $\Gamma$.
Distance regular graphs and generalised quadrangles are edge-regular graphs satisfying further regularity conditions, for their definitions, see Brouwer and Haemers~\cite{brou:spec11} (see below for the definition of a distance regular graph).
Brouwer~\cite{Brouwer84} gave a construction for antipodal distance regular graphs from generalised quadrangles having a spread of regular cliques.
Inspired by Brouwer, our construction is a generalisation of the other direction, i.e., we construct graphs from antipodal distance-regular graphs of diameter three. 

In Section~\ref{sec:construction} we present our construction and in Section~\ref{sec:3} we show that non-strongly-regular, edge-regular graphs having $1$-regular cliques must have at least $24$ vertices.

\section{Construction} 
\label{sec:construction}

Let $\Gamma$ be a graph with diameter $d$.
We call $\Gamma$ \textbf{distance-regular} if, for any two vertices $x, y \in V(\Gamma)$, the number of vertices at distance $i$ from $x$ and distance $j$ from $y$ depends only on $i$, $j$, and the distance from $x$ to $y$.
It is clear that distance regular graphs are edge-regular.

Let $\Gamma$ be a 
graph of diameter $d$.
For a given $t$, make $t$ copies $\Gamma^{(1)},\dots,\Gamma^{(t)}$ of $\Gamma$.
For each vertex $x \in V(\Gamma)$,  denote by $\Gamma_i(x)$ the set of vertices at distance $i$ from $x$ and denote by $x_j$ the corresponding copy of $x$ in the graph $\Gamma^{(j)}$.
Define the sets
\begin{align*}
	E_1(\Gamma) &= \{ \{x_i,y_j \} \; : \; x \in V(\Gamma), \, y \in \Gamma_d(x), \text{ and }  i, j \in \{1,\dots,t\} \} \\
	E_2(\Gamma) &= \{ \{x_i,x_j \} \; : \; x \in V(\Gamma) \text{ and } i \ne j \}.
\end{align*}
Let $\hat \Gamma$ denote the disjoint union of the graphs $\Gamma^{(1)},\dots,\Gamma^{(t)}$.
Then define $F_t(\Gamma)$ to be the graph with vertex set $V(F_t(\Gamma)) = V(\hat \Gamma)$ and edge set
\[
	E(F_t(\Gamma)) = E(\hat \Gamma) \cup E_1(\Gamma) \cup E_2(\Gamma).
\]

\begin{theorem}\label{thm:main}
	Let $\Gamma$ be an $a$-antipodal distance-regular graph of diameter $3$ with edge-regular parameters $(v, k, \lambda)$ such that $a$ is a proper divisor of $\lambda + 2$.
	Then
	\begin{enumerate}
		\item $F_{\frac{\lambda+2}{a}}(\Gamma)$ has a spread of $1$-regular cliques each of size $\lambda+2$;
		\item $F_{\frac{\lambda+2}{a}}(\Gamma)$ is $(v(\lambda+2)/a, k+\lambda+1, \lambda)$-edge-regular;
		\item $F_{\frac{\lambda+2}{a}}(\Gamma)$ is not strongly regular.
	\end{enumerate}
\end{theorem}
\begin{proof}
	Let $t = (\lambda+2)/a$.
	\begin{enumerate}
%
		\item Since $\Gamma$ is $a$-antipodal, its vertex set can be partitioned into $v/a$ $a$-subsets of $V(\Gamma)$ such that each subset contains a vertex and all its antipodes.
				 For each part $P$ in the partition, take a vertex $x \in P$ and define the set 
				 $$\mathcal C_x = \{ x_i \; : \; i \in \{ 1,\dots,t \} \} \cup \bigcup_{y\in \Gamma_3(x)} \{ y_i \; : \; i \in \{ 1,\dots,t \} \} \subset V(F(\Gamma)).$$
				 It is clear that each $\mathcal C_x$ is a clique in $F(\Gamma)$ of size $\lambda+2$ and that these cliques partition the vertex set of $F(\Gamma)$.
				 To see that these cliques are $1$-regular, consider a vertex $z$ not in the clique $\mathcal C_x$.
				 Let $\Gamma^{(i)}$ be the copy of $\Gamma$ containing $z$.
				 Since $\Gamma$ is distance-regular and has diameter $3$, the vertex $z$ must be adjacent to precisely one vertex in the set $\{ x_i \} \cup \{ y_i \; : \; y \in \Gamma_3(x) \}$.
				 Therefore $\mathcal C_x$ is $1$-regular.

				\item It is clear that $F(\Gamma)$ has $vt = v(\lambda+2)/a$ vertices.
				Let $x$ be a vertex of $F(\Gamma)$ inside the $i$th copy $\Gamma^{(i)}$ of $\Gamma$.
				Then $x$ is adjacent to $k+a-1$ vertices inside $\Gamma^{(i)}$ and $x$ is adjacent to the vertices $x_j$ and $y_j$ for all $j \ne i$ and $y \in \Gamma_3(x)$.
				Hence $x$ has valency $k+a-1+a(t-1) = k+\lambda+1$.
				Now suppose that $y$ is adjacent to $x$.
				If $y$ is in the clique $\mathcal C_x$ then, since $\mathcal C_x$ is $1$-regular, $y$ and $x$ have $\lambda$ common neighbours.
				Otherwise $y$ must be a vertex of $\Gamma^{(i)}$, where the number of common neighbours of $y$ and $x$ are equal to the number of common neighbours of adjacent vertices in $\Gamma$, which is $\lambda$.
		
		\item Let $x$ be a vertex of $F_t(\Gamma)$ inside the $i$th copy $\Gamma^{(i)}$ of $\Gamma$.
		Form the set $\mu(\Gamma) = \{ \nu_{y,z} \; : \; y, z \in V(\Gamma) \text{ and } y \not \sim z \}$, where $\nu_{y,z}$ denotes the number of common neighbours of $y$ and $z$.
		Since $\Gamma$ is not strongly regular, the set $\mu(\Gamma)$ must have at least $2$ elements.
		Furthermore, since $\Gamma$ has diameter $3$, we see that $0 \in \mu(\Gamma)$.
		Let $\eta \in \mu(\Gamma)$ with $\eta \ne 0$.
		Consider the vertices $y$ and $z$, where $y$ is in $\Gamma^{(i)}$ such that $y \not \sim x$ and $\nu_{x,y} = \eta$ and $z$ is in $\Gamma^{(j)}$ with $j \ne i$ and $z \not \sim x$.
		The number of common neighbours of $x$ and $y$ is $\eta + 2$ and the number of common neighbours of $x$ and $z$ is $2$.
		Hence $F_t(\Gamma)$ is not strongly regular. \qedhere
	\end{enumerate}
\end{proof}

A \textbf{Taylor graph} is a $2$-antipodal distance-regular graph of diameter $3$ (for a proper definition, see Brouwer, Cohen, and Neumaier~\cite[Page 13]{bcn89}).

\begin{example}
	Let $\Gamma$ be a Taylor graph with edge-regular parameters $(v,k,\lambda)$.
	It is known~\cite[Theorem 1.5.3]{bcn89} that $\lambda$ is even.
	By Theorem~\ref{thm:main}, the graph $F_{\lambda/2+1}(\Gamma)$ is a non-strongly-regular $(v(\lambda+2)/2,k+\lambda+1,\lambda)$-edge-regular graph having a $1$-regular clique.
	The smallest example of this family is the icosahedral graph $\mathcal P$, which has parameters $(12,5,2)$.
	The graph $F_2(\mathcal P)$ is a non-strongly-regular $(24,8,2)$-edge-regular graph having a $1$-regular clique, furthermore, $F_2(\mathcal P)$ is isomorphic to one of the four examples of Goryainov and Shalaginov~\cite{Goryainov14:Deza}.
\end{example}

We can also use other constructions of antipodal distance-regular graphs of diameter $3$ due to Brouwer, Hensel, and Mathon (see Godsil and Hensel~\cite{GodHen92} or Brouwer, Cohen, and Neumaier~\cite[Page 385]{bcn89}).
These constructions produce $a$-antipodal edge-regular graphs satisfying $\lambda +2 \equiv 0 \pmod a$ with $a \geqslant 3$.


\emph{Note added in proof:}
Sergey Goryainov observed that one can generalise the construction given in this section. One does not need an isomorphism between the antipodal graphs, a bijection between the fibre classes suffices. 

\section{At least $24$ vertices for $1$-regular cliques} 
\label{sec:3}


In the remainder of this note, we prove the following result.

\begin{theorem}\label{thm:24small}
	Let $\Gamma$ be an edge-regular graph with a $1$-regular clique that is not strongly regular.
	Then $\Gamma$ has at least $24$ vertices.
\end{theorem}

First we establish a lower bound on the vertex degree.
For a graph $\Gamma$ and a vertex $x \in V(\Gamma)$, let $\Gamma(x)$ denote the set of neighbours of $x$.
The $q \times q$ \textbf{grid} (also known as the \emph{square lattice graph}) is defined to be the Cartesian product of two complete graphs of order $q$.
It is well-known~\cite{bcn89} that the $q\times q$~grid is strongly regular with parameters $(q^2,2(q - 1),q - 2,2)$.

\begin{lemma}\label{lem:kbound}
	Let $\Gamma$ be a non-complete $k$-regular edge-regular graph having a $1$-regular clique of order $c$.
	Then $k \geqslant 2(c-1)$.
	In the case of equality, $\Gamma$ is the $c \times c$~grid and is thus strongly regular.
\end{lemma}
\begin{proof}
	Set $m = k/(c-1)$.
	Since $\Gamma$ is not complete, we have $m > 1$.
	Let $\mathcal C$ be a regular clique of order $c$.
	Let $x$ be a vertex in $\mathcal C$.
	Since there are no edges between $\Gamma(x) \cap \mathcal C$ and $\Gamma(x) \backslash \mathcal C$, we find that $\lambda = c-2$.
	
	Now suppose $y$ is a vertex adjacent to $x$ but not in $\mathcal C$.
	Note that $x$ has $k-(c-1)=(m-1)(c-1)$ neighbours outside of $\mathcal C$.
	Hence the number of common neighbours of $x$ and $y$ is at most $(m-1)(c-1) -1$.
	Therefore $c-2 \leqslant (m-2) (c-1)-1$.
	Again, since $\Gamma$ is not complete, we have $c-1 \geqslant 1$.
	Hence we must have $m \geqslant 2$.
	
	In the case of equality, we see that $y$ is adjacent to every neighbour of $x$ outside $\mathcal C$.
	Furthermore, the subgraph induced on the neighbourhood of $x$ is the disjoint union of two complete graphs each of order $c-1$.
	Therefore, $\Gamma$ is the Cartesian product of two complete graphs of order $c$, i.e., the $c \times c$~grid.
\end{proof}

Next, we need a lower bound on the size of a regular clique.

\begin{proposition}[{\cite[Proposition 5.2]{Greaves:2017qy}}]\label{pro:cliqueBound}
	Let $\Gamma$ be an edge-regular graph having a regular clique.
	Suppose that $\Gamma$ is not strongly regular.
	Then $\Gamma$ has a regular clique of order at least $4$.
\end{proposition}

The final ingredient is a nonexistence result for a graph on $20$ vertices.

\begin{proposition}\label{pro:no20}
	There does not exist an edge-regular graph with parameters $(20,7,2)$ and a $1$-regular clique.
\end{proposition}
\begin{proof}
	Suppose, for a contradiction, there does exist such a graph $\Gamma$.
	By Proposition~\ref{pro:cliqueBound}, $\Gamma$ must have a regular clique $\mathcal C$ of size at least $4$.
	Moreover, since $\lambda = 2$, the clique $\mathcal C$ must have size $4$.  
	Let $x \in \mathcal C$.
	The subgraph induced on $\Gamma(x)$ is the disjoint union of a $3$-cycle $T$ and a $4$-cycle $C$.
	Set $K = \{x\} \cup V(T)$ and let $\Delta$ be the subgraph induced on the vertices $V(\Gamma) \backslash K$.
	Observe that, since $K$ is a $1$-regular clique of $\Gamma$, the subgraph $\Delta$ is $6$-regular.
	Furthermore, observe that each of the $16$ pairs of adjacent vertices in $\bigcup_{x \in K} \Gamma(x) \backslash K$ has a common vertex in $K$.
	Hence there are $16$ edges in $\Delta$ that are each contained in precisely one triangle and the remaining $32$ edges are in precisely $2$ triangles.
	Therefore $\Delta$ has $(16+2\cdot 32)/3$ triangles.
	Since this number is not an integer, we establish a contradiction.
\end{proof}

Now the proof of Theorem~\ref{thm:24small}.
\begin{proof}[Proof of Theorem~\ref{thm:24small}]
	By Proposition~\ref{pro:cliqueBound}, $\Gamma$ must have a regular clique $\mathcal C$ of size $c \geqslant 4$.
	Furthermore, by Lemma~\ref{lem:kbound}, the degree of the vertices of $\Gamma$ is at least $7$.
	If $k \geqslant 8$ then, since $\mathcal C$ is $1$-regular, $\Gamma$ must have at least $24$ vertices.
	It therefore suffices to consider the case when $k = 7$ and $c=4$.
	In this case, $\Gamma$ must be edge-regular with parameters $(20,7,2)$.
	But by Proposition~\ref{pro:no20}, no such graph exists.
\end{proof}

\section*{Acknowledgement} 
\label{sec:acknowledgement}

We thank Rhys Evans, Leonard Soicher, and an anonymous referee for their comments, which greatly improved this paper.



%

\bibliographystyle{myplain}
\bibliography{sbib}

\begin{thebibliography}{1}

\bibitem{Brouwer84}
A.~E. Brouwer.
\newblock Distance regular graphs of diameter {$3$}\ and strongly regular
  graphs.
\newblock {\em Discrete Math.}, 49(1):101--103, 1984.

\bibitem{bcn89}
A.~E. Brouwer, A.~M. Cohen, and A. Neumaier.
\newblock {\em Distance-regular graphs}, volume~18 of {\em Ergebnisse der
  Mathematik und ihrer Grenzgebiete (3) [Results in Mathematics and Related
  Areas (3)]}.
\newblock Springer-Verlag, Berlin, 1989.

\bibitem{brou:spec11}
A.~E. Brouwer and W.~H. Haemers.
\newblock {\em Spectra of Graphs}.
\newblock Universitext, Springer, New York (2012), 2012.

\bibitem{EGP}
R. Evans, S. Goryainov, and D. Panasenko.
\newblock The smallest strictly {N}eumaier graph and its generalisations.
\newblock arXiv:1809.03417.

\bibitem{GodHen92}
C.~D. Godsil and A.~D. Hensel.
\newblock Distance regular covers of the complete graph.
\newblock {\em J. Combin. Theory Ser. B}, 56(2):205--238, 1992.

\bibitem{Goryainov14:Deza}
S. Goryainov and L. Shalaginov.
\newblock {C}ayley-{D}eza graphs with fewer than 60 vertices (in {R}ussian).
\newblock {\em Siberian Electronic Mathematical Reports}, 11(268--310), 2014.

\bibitem{Greaves:2017qy}
G.~R.~W. Greaves and J.~H. Koolen.
\newblock Edge-regular graphs with regular cliques.
\newblock {\em European J. Combin.}, 71:194--201, 2018.

\bibitem{Neu:regCliques}
A. Neumaier.
\newblock Regular cliques in graphs and special $1\frac{1}{2}$-designs.
\newblock In {\em Finite geometries and designs, Proc. 2nd Isle of Thorns Conf.
  1980}, volume~49 of {\em Lect. Note Ser.}, pages 244--259. Lond. Math. Soc.,
  1981.

\end{thebibliography}

\end{document}